\documentclass[12pt]{article}
\usepackage{amsmath,amssymb,amsthm}
\usepackage{float}

\DeclareMathOperator{\Cay}{Cay}

\newtheorem{thm}{Theorem}[section]
\newtheorem{lem}[thm]{Lemma}
\newtheorem{prop}[thm]{Proposition}

\theoremstyle{definition}
\newtheorem{rem}[thm]{Remark}
\newtheorem{example}[thm]{Example}

\begin{document}
\date{}
\begin{center}
{\bf\Large Existence of regular sets in circulant quartic graphs: a complete characterization}
\vspace{1.0cm}

{\small A.~Abdollahi$^{a}$, J.~Bagherian$^{a,*}$, F.~Jafari$^{a}$, M.~Khatami$^{a}$, Z.~Shokoohi$^{a}$, R.~Sobhani$^{b}$\\[3mm]
$^a$ Department of Pure Mathematics, Faculty of Mathematics and Statistics, University of Isfahan, Isfahan 81746-73441, Iran\\[1mm]
$^b$ Department of Applied Mathematics and Computer Science, Faculty of Mathematics and Statistics, University of Isfahan, Isfahan 81746-73441, Iran\\[2mm]
{\bf Emails:} a.abdollahi@math.ui.ac.ir, bagherian@sci.ui.ac.ir, math\_fateme@yahoo.com, m.khatami@sci.ui.ac.ir, zahra.shokoohi@sci.ui.ac.ir, r.sobhani@sci.ui.ac.ir\\[1mm]
$^*$Corresponding author.}
\vspace{1cm}
\end{center}

\begin{abstract}
Let $\Gamma=(V,E)$ be a graph and let $a,b$ be nonnegative integers. A nonempty proper subset $C\subset V$ is an \emph{$(a,b)$-regular set} if each vertex of $C$ has exactly $a$ neighbors in $C$ and each vertex outside $C$ has exactly $b$ neighbors in $C$. Perfect codes and total perfect codes are the $(0,1)$- and $(1,1)$-regular sets. For every connected circulant quartic graph $\operatorname{Cay}(\mathbb{Z}_n,\{x^{\pm k},x^{\pm\ell}\})$ and every pair $(a,b)$ we decide, in terms of $n$, $k$ and $\ell$, whether an $(a,b)$-regular set exists. Existence is proved by exhibiting unions of residue classes; non-existence comes from a Fourier-analytic obstruction combined with the fact that certain sums of roots of unity cannot vanish.
\end{abstract}


\section{Introduction}

All graphs in this paper are finite and simple. For a vertex $v$ of a graph $\Gamma$, $N(v)$ is the set of its neighbors. Given a group $G$ with identity $e$ and an inverse-closed subset $S\subseteq G\setminus\{e\}$, the Cayley graph $\Gamma=\Cay(G,S)$ has vertex set $G$, and $g,h$ are adjacent exactly when $g^{-1}h\in S$ (see \cite{BiggsAGT}). It is $|S|$-regular, and it is connected if and only if $S$ generates $G$. A Cayley graph of a cyclic group is a \emph{circulant}. We write $\mathbb{Z}_n=\langle x\rangle$, and call $\Gamma$ a \emph{circulant quartic graph} when $|S|=4$, that is, $S=\{x^{k},x^{-k},x^{\ell},x^{-\ell}\}$ with four distinct elements. We write $\Gamma=\Cay(\mathbb{Z}_n,\{x^{\pm k},x^{\pm\ell}\})$; the neighbors of $x^m$ are then $x^{m\pm k}$ and $x^{m\pm\ell}$.

A subset $C$ of $V(\Gamma)$ is a \emph{perfect code} if it is independent and every vertex outside $C$ has exactly one neighbor in $C$, and a \emph{total perfect code} if every vertex of $\Gamma$ has exactly one neighbor in $C$. Perfect codes are also called independent perfect dominating sets \cite{Lee} or efficient dominating sets \cite{DS}; total perfect codes are efficient open dominating sets \cite{GSS,HHS}. Biggs \cite{Biggs} introduced perfect codes in graphs as a generalization of Hamming's error-correcting codes, and perfect and total perfect codes in Cayley graphs and Cayley sum graphs have been studied at length \cite{AAR,CYS,CWX,HXS,LLL,MFW,MWWZ,WWSZ,Zhang,Zhou}.

For nonnegative integers $a$ and $b$, a nonempty proper subset $C\subset V(\Gamma)$ is an \emph{$(a,b)$-regular set} if $|N(v)\cap C|=a$ for all $v\in C$ and $|N(v)\cap C|=b$ for all $v\in V(\Gamma)\setminus C$ (these are also called $(\kappa,\tau)$-regular sets; see \cite{Cardoso}). Perfect codes are exactly the $(0,1)$-regular sets and total perfect codes are exactly the $(1,1)$-regular sets, so regular sets provide a common framework for both notions. They have been investigated in Cayley graphs and Cayley sum graphs in \cite{WXZ,KAK,SKA,WXS,WXSJ,WXSZ,ZY,ZZ}.

\subsection{Related work}

On circulant graphs the cases $(0,1)$ and $(1,1)$ have received most of the attention. Kumar and MacGillivray \cite{KM} used wreath products to build infinitely many circulant graphs with $(0,1)$-regular sets whose elements are not equally spaced, and showed that for large degree a $(0,1)$-regular set is either equally spaced or the graph has a wreath-product structure. Deng \cite{Deng} characterized the existence of $(0,1)$-regular sets in $\Cay(\mathbb{Z}_n,S)$ when $n/(|S|+1)$ is prime. Deng, Sun, Liu and Wang \cite{dengg} extended the results of Obradovi\'c, Peters and Ru\v zi\'c \cite{OPR} on connected circulant graphs of degrees $3$ and $4$ to connected non-complete circulant graphs of degree $p-1$, $pq-1$ and $p^m-1$ ($p,q$ primes), giving necessary and sufficient conditions in terms of the connection set. Using cyclotomic polynomials, Feng, Huang and Zhou \cite{rWXSJ} settled $(0,1)$-regular sets in circulant graphs of degree $p-1$ and $p^\ell-1$ and $(1,1)$-regular sets in those of degree $p$ and $p^\ell$; Wang, Serra, Xu and Zhou \cite{zawxj} treated degree $p^\ell-1$ for $(0,1)$-regular sets, with a construction and a lower bound for their number. \c{C}al\i\c{s}kan, Miklavi\v{c} and \"Ozkan \cite{CMO} characterized the cubic and quartic Cayley graphs of abelian groups admitting $(0,1)$-regular sets, and Dong, Yang and Dong \cite{DYD} classified the $(0,1)$-regular sets in quartic Cayley graphs of generalized dihedral groups. Cameron, Yap and Zhou \cite{CYS} gave necessary and sufficient conditions for $(1,1)$-regular sets in circulant graphs and showed that, when they exist, certain subgroups of $\mathbb{Z}_n$ are such sets. Kwon, Lee and Sohn \cite{kls} classified the $(0,1)$-regular sets in circulant graphs of degree $5$. On the complexity side, Bakker and van Leeuwen \cite{Bakker} proved that deciding whether a graph has a $(0,1)$-regular set, or a $(1,1)$-regular set, is NP-complete.

For all pairs $(a,b)$ at once, Hao and Wang \cite{HW} characterized the existence of $(a,b)$-regular sets in the connected cubic circulant graphs $\Cay(\mathbb{Z}_n,\{x^{\pm k},x^{n/2}\})$, $n>4$, and showed in particular that $(0,2)$- and $(1,3)$-regular sets never occur there. For quartic circulants, two special cases were known. Obradovi\'c, Peters and Ru\v zi\'c \cite{OPR} determined when $\Cay(\mathbb{Z}_n,\{x^{\pm k},x^{\pm\ell}\})$ has an efficient dominating set, that is, a $(0,1)$-regular set; their condition appears as Case~2 of Table~\ref{tab:main-characterization}. Kwon, Sohn and Chen \cite{KSC} studied efficient total dominating sets, which are the $(1,1)$-regular sets. They observed that $8\mid n$ is necessary but not sufficient; the exact criterion is Case~9 of the table.

\subsection{Main result}

In this paper we settle the question for all remaining pairs. We begin with two simple remarks that reduce the number of cases.

\begin{rem}\label{rem:connectivity}
Let $n,k,\ell$ be positive integers. The elements $x^{k},x^{-k},x^{\ell},x^{-\ell}$ of $\mathbb{Z}_n=\langle x\rangle$ are pairwise distinct if and only if $2k\not\equiv0$, $2\ell\not\equiv0$ and $k\not\equiv\pm\ell\pmod n$ (these conditions also exclude $x^k=1$ and $x^\ell=1$). In that case $\Gamma=\Cay(\mathbb{Z}_n,\{x^{\pm k},x^{\pm\ell}\})$ is connected if and only if $\gcd(n,k,\ell)=1$. The isomorphism type of $\Gamma$ is unchanged by

\begin{itemize}
\item[(E1)] interchanging $k$ and $\ell$;
\item[(E2)] replacing $k$ by $n-k$, or $\ell$ by $n-\ell$;
\item[(E3)] replacing $(k,\ell)$ by $(uk,u\ell)$ for an integer $u$ coprime to $n$ (the map $x^m\mapsto x^{um}$ is an automorphism of $\mathbb{Z}_n$);
\end{itemize}
and $k,\ell$ may be replaced by any integers congruent to them modulo $n$. If $\Gamma$ is connected, put $d=\gcd(k,\ell)$. Then $\gcd(d,n)=1$, and (E3) with $u\equiv d^{-1}\pmod n$ gives $\Gamma\cong\Cay(\mathbb{Z}_n,\{x^{\pm k/d},x^{\pm\ell/d}\})$. Hence for connected graphs we may, and do, assume $\gcd(k,\ell)=1$.
\end{rem}

\begin{rem}\label{ComplementaryRelation}
If $C$ is an $(a,b)$-regular set in a $d$-regular graph, then $V\setminus C$ is a $(d-b,d-a)$-regular set: a vertex outside $C$ has $b$ neighbors in $C$ and so $d-b$ outside it, while a vertex of $C$ has $d-a$ neighbors outside $C$. In a quartic graph an $(a,b)$-regular set therefore exists if and only if a $(4-b,4-a)$-regular set does. Moreover $a,b\le4$ necessarily. Up to this symmetry, it is enough to consider the fifteen pairs
\[
\begin{aligned}
&(0,0),(0,1),(0,2),(0,4),(1,0),(2,0),(3,0),(4,0),\\
&(1,1),(1,2),(1,3),(1,4),(2,1),(2,2),(3,1),
\end{aligned}
\]
since the remaining ten pairs are the complements $(4-b,4-a)$ of these.
\end{rem}

\begin{thm}\label{thm:main}
Let $n,k,\ell$ be positive integers with $\gcd(k,\ell)=1$ such that $x^{k},x^{-k},x^{\ell},x^{-\ell}$ are four distinct elements of $\mathbb{Z}_n=\langle x\rangle$, and let $\Gamma=\Cay(\mathbb{Z}_n,\{x^{\pm k},x^{\pm\ell}\})$. Then the existence of $(a,b)$-regular sets in $\Gamma$ is given by Table~\ref{tab:main-characterization}. Every part of every row is a statement about all pairs $(k,\ell)$ satisfying the stated congruences, and by Proposition~\ref{prop:completeness} every connected circulant quartic graph is covered by the table after applying the moves {\rm (E1)--(E3)}.
\begin{table}[H]
	\centering
	\footnotesize
	\setlength{\tabcolsep}{2pt}
	\renewcommand{\arraystretch}{0.95}
	\begin{tabular}{|c|c|p{110mm}|}
		\hline
		\textbf{Case} & \textbf{Regular set} & \textbf{Characterization} \\
		\hline
		1 & $(0,0)$ & \quad $\Gamma$ contains no $(0,0)$-regular set. \\
		\hline
		2 & $(0,1)$ & $\Gamma$ contains a $(0,1)$-regular set if and only if $5\mid n$, $k\pm\ell\not\equiv 0\pmod 5$ and $k,\ell\not\equiv 0 \pmod 5$ \cite{OPR}. \\
		\hline
		3 & $(0,2)$ &
		\textbf{(i)} If $k\equiv\ell\equiv1\pmod{3}$, then $\Gamma$ contains a $(0,2)$-regular set if and only if $3\mid n$. \newline
		\textbf{(ii)} If $k\equiv1\pmod{3}$ and $\ell\equiv0\pmod{3}$, or $\ell\equiv1\pmod{3}$ and $k\equiv0\pmod{3}$, then $\Gamma$ contains no $(0,2)$-regular set. \\
		\hline
		4 & $(0,4)$ &
		\textbf{(i)} If exactly one of $k,\ell$ is even, then $\Gamma$ contains no $(0,4)$-regular set. \newline
		\textbf{(ii)} If $k\equiv\ell\equiv1\pmod{2}$, then $\Gamma$ contains a $(0,4)$-regular set if and only if $2\mid n$. \\
		\hline
		5 & $(1,0)$ & \quad $\Gamma$ contains no $(1,0)$-regular set. \\
		\hline
		6 & $(2,0)$ & \quad $\Gamma$ contains no $(2,0)$-regular set. \\
		\hline
		7 & $(3,0)$ & \quad $\Gamma$ contains no $(3,0)$-regular set. \\
		\hline
		8 & $(4,0)$ & \quad $\Gamma$ contains no $(4,0)$-regular set. \\
		\hline
		9 & $(1,1)$ &
      \textbf{(i)} If $k \equiv 1 \pmod{8}$ and $\ell \equiv 1$ or $3 \pmod{8}$, then $\Gamma$ contains a $(1,1)$-regular set if and only if $n = 8m$ for some $m \in \mathbb{N}$ and either $\ell \equiv 3 \pmod{8}$, or $\ell \equiv 1 \pmod{8}$ and $\gcd(8m, |k-\ell|) = \gcd(4m, |k-\ell|)$ \cite{KSC}. \newline
      \textbf{(ii)} If exactly one of $k,\ell$ is even, then $\Gamma$ contains no $(1,1)$-regular set. \\
      \hline
		10 & $(1,2)$ &
		\textbf{(i)} If $k\equiv1\pmod{5}$ and $\ell\equiv2\pmod{5}$, or $\ell\equiv1\pmod{5}$ and $k\equiv2\pmod{5}$, then $\Gamma$ contains a $(1,2)$-regular set if and only if $5\mid n$. \newline
		\textbf{(ii)} If $k\equiv\ell\equiv1\pmod{5}$, then $\Gamma$ contains no $(1,2)$-regular set. \newline
		\textbf{(iii)} If $k\equiv1\pmod{5}$ and $\ell\equiv0\pmod{5}$, or $\ell\equiv1\pmod{5}$ and $k\equiv0\pmod{5}$, then $\Gamma$ contains no $(1,2)$-regular set. \\
		\hline
		11 & $(1,3)$ &
		\textbf{(i)} If $k$ is odd and $\ell\equiv2\pmod{4}$, or $\ell$ is odd and $k\equiv2\pmod{4}$, then $\Gamma$ contains a $(1,3)$-regular set if and only if $4\mid n$. \newline
		\textbf{(ii)} If $k$ is odd and $\ell\equiv0\pmod{4}$, or $\ell$ is odd and $k\equiv0\pmod{4}$, then $\Gamma$ contains no $(1,3)$-regular set. \newline
		\textbf{(iii)} If $k\equiv\ell\equiv1\pmod{2}$, then $\Gamma$ contains no $(1,3)$-regular set. \\
		\hline
		12 & $(1,4)$ & \quad $\Gamma$ contains no $(1,4)$-regular set. \\
		\hline
		13 & $(2,1)$ &
		\textbf{(i)} If $k\equiv1\pmod{3}$ and $\ell\equiv0\pmod{3}$, or $\ell\equiv1\pmod{3}$ and $k\equiv0\pmod{3}$, then $\Gamma$ contains a $(2,1)$-regular set if and only if $3\mid n$. \newline
		\textbf{(ii)} If $k\equiv\ell\equiv1\pmod{3}$, then $\Gamma$ contains no $(2,1)$-regular set. \\
		\hline
		14 & $(2,2)$ &
		\textbf{(i)} If exactly one of $k,\ell$ is even, then $\Gamma$ contains a $(2,2)$-regular set if and only if $2\mid n$. \newline
		\textbf{(ii)} If $k\equiv\ell\equiv1\pmod{2}$, then $\Gamma$ contains a $(2,2)$-regular set if and only if $4\mid n$. \\
		\hline
		15 & $(3,1)$ &
		\textbf{(i)} If $k$ is odd and $\ell\equiv0\pmod{4}$, or $\ell$ is odd and $k\equiv0\pmod{4}$, then $\Gamma$ contains a $(3,1)$-regular set if and only if $4\mid n$. \newline
		\textbf{(ii)} If $k$ is odd and $\ell\equiv2\pmod{4}$, or $\ell$ is odd and $k\equiv2\pmod{4}$, then $\Gamma$ contains no $(3,1)$-regular set. \newline
		\textbf{(iii)} If $k\equiv\ell\equiv1\pmod{2}$, then $\Gamma$ contains no $(3,1)$-regular set. \\
		\hline
	\end{tabular}
	\caption{Regular sets in connected circulant quartic graphs.}
	\label{tab:main-characterization}
\end{table}
\end{thm}

\begin{rem}
Each row needs a divisibility condition $m\mid n$ for a regular set to exist (see the proof of Proposition~\ref{prop:completeness}); when $m\nmid n$, every part of that row asserts non-existence. Thus the congruences of $k,\ell$ modulo $3,4,5,8$ occurring in the table can be read off from any representatives of $k,\ell$ modulo $n$: they matter only when the relevant modulus divides $n$, and then they are independent of the representatives.
\end{rem}

\begin{prop}\label{prop:completeness}
Let $\Gamma$ be a connected circulant quartic graph on $n$ vertices and let $(a,b)$ be one of the fifteen pairs of Remark~\ref{ComplementaryRelation}. Then at least one of the following holds.
\begin{enumerate}
\item[(a)] $(a,b)\in\{(0,0),(1,0),(2,0),(3,0),(4,0),(1,4)\}$, and $\Gamma$ has no $(a,b)$-regular set.
\item[(b)] $m\nmid n$, where $m=5$ for $(a,b)\in\{(0,1),(1,2)\}$, $m=3$ for $(0,2),(2,1)$, $m=2$ for $(0,4),(1,3),(2,2),(3,1)$ and $m=8$ for $(1,1)$; in this case $\Gamma$ has no $(a,b)$-regular set.
\item[(c)] $\Gamma\cong\Cay(\mathbb{Z}_n,\{x^{\pm k},x^{\pm\ell}\})$ for some positive integers $k,\ell$ with $\gcd(k,\ell)=1$ satisfying the hypothesis of some part of the row of Table~\ref{tab:main-characterization} belonging to $(a,b)$ (or the row has no hypothesis).
\end{enumerate}
\end{prop}

The proposition is proved in Section~2. The rows of the table are established as follows: Cases 1 and 5--8 by Lemma~\ref{lem:no-a0}; Cases 2 and 9 are due to \cite{OPR} and \cite{KSC} (Case 9(ii) needs a short argument, given in Section~2); Case 14 is proved in Section~3, Case 4 in Section~4, Case 12 in Section~5, and Cases 3, 10, 11, 13, 15 in Section~6. For the last group we combine explicit constructions with a spectral argument based on cyclotomic polynomials.

\section{Preliminaries}

We begin with three elementary facts that hold in any regular graph.

\begin{lem}\label{lem:count}
Let $C$ be an $(a,b)$-regular set in a $d$-regular graph with $n$ vertices. Then $(d+b-a)\,|C|=bn$.
\end{lem}
\begin{proof}
Count the pairs $(v,c)$ with $c\in C$ and $v\sim c$ in two ways. Each $c\in C$ lies in $d$ such pairs, giving $d|C|$. Each $v\in C$ lies in $a$ pairs and each $v\notin C$ in $b$ pairs, giving $a|C|+b(n-|C|)$.
\end{proof}

\begin{lem}\label{lem:no-a0}
A connected graph has no $(a,0)$-regular set, for any $a\ge0$.
\end{lem}
\begin{proof}
If $C$ were an $(a,0)$-regular set, no edge would join $C$ to its complement. As $C\neq\emptyset$, connectedness would force $C=V$, but $C$ is proper.
\end{proof}

\noindent This gives Cases 1 and 5--8 of Theorem~\ref{thm:main}, namely the pairs $(a,0)$ with $0\le a\le4$.

\begin{lem}\label{lem:11}
If a $4$-regular graph on $n$ vertices has a $(1,1)$-regular set $C$, then $n=4|C|$ and $|C|$ is even; in particular $8\mid n$.
\end{lem}
\begin{proof}
Every vertex has exactly one neighbor in $C$, so the neighborhoods $N(c)$, $c\in C$, partition the vertex set into blocks of size $4$; hence $n=4|C|$. Moreover, each $c\in C$ has exactly one neighbor $c'$ in $C$, and then $c$ is the unique neighbor of $c'$ in $C$. Thus $c\mapsto c'$ is a fixed-point-free involution of $C$, and $|C|$ is even.
\end{proof}

We shall also use the following fact from \cite[Proposition~1]{KSC}.

\begin{lem}\label{gcd}
Let $z\mid n$ and $\gcd(s,z)=1$. Then there is an integer $y$ with $\gcd(y,n)=1$ and $sy\equiv1\pmod z$.
\end{lem}

\begin{proof}[Proof of Proposition~\ref{prop:completeness}]
First let $(a,b)$ be one of the pairs in (b) and suppose that $C$ is an $(a,b)$-regular set; we show that $m\mid n$, which proves the non-existence claim in (b). Lemma~\ref{lem:count} with $d=4$ gives $(4+b-a)|C|=bn$. For $(0,1)$ this reads $5|C|=n$, and for $(1,2)$ it reads $5|C|=2n$; in both cases $5\mid n$. For $(2,1)$ we get $3|C|=n$ and for $(0,2)$ we get $6|C|=2n$, so $3\mid n$. For $(0,4)$, $(1,3)$, $(3,1)$ and $(2,2)$ we get $8|C|=4n$, $6|C|=3n$, $2|C|=n$ and $4|C|=2n$ respectively, so $n=2|C|$ is even. Finally $8\mid n$ for $(1,1)$ by Lemma~\ref{lem:11}.

It remains to show that (c) holds when $m\mid n$ and $(a,b)\notin\{(0,0),(1,0),(2,0),(3,0),(4,0),(1,4)\}$. By Remark~\ref{rem:connectivity} we may write $\Gamma=\Cay(\mathbb{Z}_n,\{x^{\pm k},x^{\pm\ell}\})$ with $\gcd(k,\ell)=1$, and we must check that, after the moves (E1)--(E3), some part of the row for $(a,b)$ applies. Since $\gcd(k,\ell)=1$, the integers $k,\ell$ are not both even.

\emph{Rows $(0,1)$, $(0,4)$, $(1,3)$, $(2,2)$, $(3,1)$.} The row $(0,1)$ has no hypothesis. In the other four rows the hypotheses concern only the parity of $k,\ell$ and, when exactly one of them is even, its residue modulo $4$; every pair $(k,\ell)$ falls under one of them.

\emph{Rows $(0,2)$, $(2,1)$.} Here $3\mid n$. If $k\equiv2\pmod3$ we replace $k$ by $n-k\equiv1\pmod 3$ (E2), and likewise for $\ell$. Then $k,\ell\in\{0,1\}\pmod3$, and not both are $0$ because $\gcd(n,k,\ell)=1$. The possible residue pairs $(1,1),(1,0),(0,1)$ are exactly the hypotheses of the rows (we may restore $\gcd(k,\ell)=1$ by adding multiples of $n$, which does not change residues modulo $3$).

\emph{Row $(1,2)$.} Here $5\mid n$. By (E2) we may assume $k,\ell\in\{0,1,2\}\pmod5$, not both $0$. If the residues are $(0,2)$, $(2,0)$ or $(2,2)$, apply (E3) with $u$ coprime to $n$ and $u\equiv3\pmod5$ (it exists by Lemma~\ref{gcd}, since $2\cdot3\equiv1$); this turns $2$ into $1$. We end with one of $(1,0),(0,1),(1,1),(1,2),(2,1)$, which are the hypotheses of the row.

\emph{Row $(1,1)$.} Here $8\mid n$. If exactly one of $k,\ell$ is even, part (ii) of the row applies. Otherwise both are odd. By Lemma~\ref{gcd} there is $u$ coprime to $n$ with $uk\equiv1\pmod8$; applying (E3) and then, if necessary, (E2) to $\ell$ (which maps the residues $5,7$ modulo $8$ to $3,1$), we reach $k\equiv1$ and $\ell\equiv1$ or $3\pmod8$.
\end{proof}

\begin{proof}[Proof of Case 9(ii) of Theorem~\ref{thm:main}]
Let $k$ be odd and $\ell$ even, and suppose $\Gamma$ has a $(1,1)$-regular set. By Lemma~\ref{lem:11}, $8\mid n$. By Lemma~\ref{gcd} and (E3) we may assume $k\equiv1\pmod8$, and then $\ell\equiv0,2,4,6\pmod8$; replacing $\ell$ by $n-\ell$ if $\ell\equiv6$, we may take $\ell\equiv0,2,4\pmod8$. Kwon, Sohn and Chen \cite{KSC} proved that for $n=8m$, $k\equiv1\pmod8$ and $\ell\equiv0,1,2,3,4\pmod8$ the graph has a $(1,1)$-regular set if and only if $\ell\equiv3\pmod8$, or $\ell\equiv1\pmod8$ and $\gcd(8m,|k-\ell|)=\gcd(4m,|k-\ell|)$. Neither condition holds for even $\ell$, so there is no $(1,1)$-regular set.
\end{proof}

\section{$(2,2)$-regular sets}

\begin{proof}[Proof of Case 14 of Theorem~\ref{thm:main}]
Let $C$ be a $(2,2)$-regular set. By Lemma~\ref{lem:count}, $2|C|=n$, so $n$ is even; this is the necessity in part (i). Suppose that $k$ is odd and $\ell$ is even (the other case is symmetric) and $n$ is even. Take $C=\langle x^2\rangle=\{x^m:\ m\text{ even}\}$. Call $x^m$ even or odd according to the parity of $m$ (well defined since $n$ is even). A vertex $x^m$ has neighbors $x^{m\pm\ell}$ of the same parity and $x^{m\pm k}$ of the opposite parity. Hence an even vertex has exactly two neighbors in $C$ (namely $x^{m\pm\ell}$) and an odd vertex also has exactly two (namely $x^{m\pm k}$). So $C$ is $(2,2)$-regular, proving (i).

Now let $k,\ell$ both be odd. We still have $n$ even, and all four neighbors of any vertex have parity opposite to it. Let $c_0$ and $c_1$ be the numbers of even and odd vertices in $C$. Counting the pairs (even vertex, odd element of $C$ adjacent to it) in two ways gives $2\cdot\frac n2=4c_1$, and similarly $2\cdot\frac n2=4c_0$. Hence $c_0=c_1=n/4$ and $4\mid n$.

Conversely, let $4\mid n$ and $C=\{x^m:\ m\bmod4\in\{0,1\}\}$. As $k$ and $\ell$ are odd, $\{\pm k\}$ and $\{\pm\ell\}$ both reduce to $\{1,3\}$ modulo $4$. So the exponents of the four neighbors $x^{m\pm k},x^{m\pm\ell}$ of $x^m$ are, modulo $4$, equal to $m+1,m+3,m+1,m+3$. If $m\equiv0$ these are $1,3,1,3$, two of which lie in $\{0,1\}$; for $m\equiv1$ they are $2,0,2,0$; for $m\equiv2$ they are $3,1,3,1$; for $m\equiv3$ they are $0,2,0,2$. In each case exactly two of the neighbors lie in $C$, so $C$ is $(2,2)$-regular. This proves (ii).
\end{proof}

\begin{example}
In $\Cay(\mathbb{Z}_8,\{x^{\pm1},x^{\pm2}\})$ the set $\{1,x^2,x^4,x^6\}$ is $(2,2)$-regular. The graph $\Cay(\mathbb{Z}_8,\{x^{\pm1},x^{\pm3}\})$ is the complete bipartite graph $K_{4,4}$ (its parts are the even and the odd powers of $x$), and $\{1,x,x^2,x^3,x^4,x^5\}$ is a $(3,3)$-regular set in it; its complement $\{x^6,x^7\}$ is $(1,1)$-regular, in accordance with Case 9(i).
\end{example}

\section{$(0,4)$-regular sets}

Here we show that, in a connected $d$-regular graph, $(0,d)$-regular sets are precisely the two parts of a bipartition.

\begin{lem}\label{lem:bip-regular}
Let $\Gamma$ be a connected $d$-regular graph with $d\ge1$. Then $\Gamma$ has a $(0,d)$-regular set if and only if $\Gamma$ is bipartite, and the $(0,d)$-regular sets are then exactly the two parts of the bipartition.
\end{lem}
\begin{proof}
Let $C$ be $(0,d)$-regular. Then $C$ is independent, and every vertex outside $C$ has all its $d$ neighbors in $C$, so $V\setminus C$ is independent too and $(C,V\setminus C)$ is a bipartition. Conversely, if $(X,Y)$ is a bipartition, then $X$ is independent and each vertex of $Y$ has all $d$ neighbors in $X$, so $X$, and likewise $Y$, is $(0,d)$-regular. A connected bipartite graph has a unique bipartition, which gives the last claim.
\end{proof}

\begin{lem}\label{lem:bip-cyclic}
Let $\Gamma=\Cay(\mathbb{Z}_n,S)$ be connected, where $\mathbb{Z}_n=\langle x\rangle$. Then $\Gamma$ is bipartite if and only if $n$ is even and every element of $S$ is an odd power of $x$.
\end{lem}
\begin{proof}
If $n$ is even and all elements of $S$ are odd powers of $x$, the even and the odd powers of $x$ form a bipartition. Conversely, let $(X,Y)$ be the bipartition of $\Gamma$, with $1\in X$. For $g\in\mathbb{Z}_n$ the map $v\mapsto vg$ is an automorphism of $\Gamma$, so it either fixes both parts or interchanges them, because the bipartition is unique. Let $h(g)\in\mathbb{Z}_2$ be $0$ in the first case and $1$ in the second; then $h:\mathbb{Z}_n\to\mathbb{Z}_2$ is a homomorphism. For $s\in S$, the map $v\mapsto vs$ sends $1\in X$ to its neighbor $s\in Y$, so $h(s)=1$. Since $S$ generates $\mathbb{Z}_n$, $h$ is onto; hence $n$ is even and $\ker h=\langle x^2\rangle$. As $h(s)=1$, each $s\in S$ is an odd power of $x$.
\end{proof}

\begin{proof}[Proof of Case 4 of Theorem~\ref{thm:main}]
By Lemmas~\ref{lem:bip-regular} and \ref{lem:bip-cyclic}, $\Gamma$ has a $(0,4)$-regular set if and only if $n$ is even and $k,\ell$ are both odd. Since $\gcd(k,\ell)=1$, the integers $k,\ell$ are not both even, so either they are both odd (part (ii)) or exactly one is even (part (i)).
\end{proof}

\section{$(1,4)$-regular sets}

\begin{thm}\label{thm:1S}
Let $G$ be a group, written multiplicatively, and let $\Gamma=\Cay(G,S)$ be a connected Cayley graph with $|S|\ge3$, where $S$ is closed under inverses and under conjugation. Then $\Gamma$ has no $(1,|S|)$-regular set.
\end{thm}
\begin{proof}
Suppose $C$ is a $(1,|S|)$-regular set. Take $g\in C$ and let $h$ be its unique neighbor in $C$; put $s=g^{-1}h\in S$. Because $|S|\ge3$ and $S$ is inverse-closed, there is $t\in S$ with $t\ne s,s^{-1}$. Consider $u=gt$ and $v=ht$. They are adjacent, since $u^{-1}v=t^{-1}st\in S$, and $u\neq v$.

We have $u\in N(g)$ and $u\ne h$ (otherwise $t=s$), so $u\notin C$ because $h$ is the only neighbor of $g$ in $C$. In the same way $v\in N(h)$ and $v\ne g$ (otherwise $t=h^{-1}g=s^{-1}$), so $v\notin C$, since $g$ is the only neighbor of $h$ in $C$. But $u\notin C$ implies that all $|S|$ neighbors of $u$ lie in $C$, whereas its neighbor $v$ does not. This is a contradiction.
\end{proof}

\begin{proof}[Proof of Case 12 of Theorem~\ref{thm:main}]
$\mathbb{Z}_n=\langle x\rangle$ is abelian, so $S$ is closed under conjugation; apply Theorem~\ref{thm:1S} with $|S|=4$.
\end{proof}

\section{The remaining cases}\label{sec:remaining}

Throughout this section $\Gamma=\Cay(\mathbb{Z}_n,\{x^{\pm k},x^{\pm\ell}\})$ is connected, so $\gcd(n,k,\ell)=1$, and $C$ is an $(a,b)$-regular set with $d:=b-a$. For $m\in\mathbb{Z}_n$ we put $c(m)=1$ if $x^m\in C$ and $c(m)=0$ otherwise, and we write $v_p$ for the $p$-adic valuation. Since the statements of Theorem~\ref{thm:main} are symmetric in $k$ and $\ell$, we may interchange them freely. We prove Cases 3, 10, 11, 13 and 15, that is, the pairs $(a,b)\in\{(0,2),(1,2),(1,3),(2,1),(3,1)\}$. Existence is shown by explicit unions of residue classes (Lemma~\ref{lem:constructions}); non-existence follows from a spectral obstruction (Lemma~\ref{lem:obstruction}) together with the non-vanishing of certain sums of roots of unity (Propositions~\ref{prop:p2} and~\ref{prop:podd}).

\subsection{Tools}

\begin{lem}\label{lem:engine}
Let $p$ be a prime, $\alpha\ge1$ with $p^{\alpha}\mid n$, and let $\omega$ be a primitive $p^{\alpha}$-th root of unity. Put $t_h=|\{x^g\in C:\ g\equiv h\pmod{p^\alpha}\}|$ and
\[
\lambda_j=d+\omega^{jk}+\omega^{-jk}+\omega^{j\ell}+\omega^{-j\ell}\qquad(1\le j\le p^\alpha-1).
\]
If $\lambda_j\ne0$ for all $1\le j\le p^\alpha-1$, then $t_h=|C|/p^\alpha$ for every $0\le h<p^\alpha$.
\end{lem}
\begin{proof}
Let $S=\{k,-k,\ell,-\ell\}\subset\mathbb{Z}_n$. Since the neighbors of $x^m$ are the vertices $x^{m+s}$, $s\in S$, the definition of an $(a,b)$-regular set says that for every $m\in\mathbb{Z}_n$
\[
\sum_{s\in S}c(m+s)=b-d\,c(m).
\]
Since $p^\alpha\mid n$, the function $m\mapsto\omega^{jm}$ is well defined on $\mathbb{Z}_n$. Put $\hat c(j)=\sum_{m\in\mathbb{Z}_n}c(m)\omega^{jm}=\sum_h t_h\omega^{jh}$. Multiply the identity by $\omega^{jm}$ and sum over $m\in\mathbb{Z}_n$:
\[
\Bigl(\sum_{s\in S}\omega^{-js}\Bigr)\hat c(j)=b\sum_{m\in\mathbb{Z}_n}\omega^{jm}-d\,\hat c(j).
\]
For $1\le j\le p^\alpha-1$ we have $\omega^j\ne1$ and $(\omega^j)^n=1$, so $\sum_m\omega^{jm}=0$; as $S=-S$, we obtain $\lambda_j\hat c(j)=0$. If every $\lambda_j$ is nonzero, then $\hat c(j)=0$ for $1\le j<p^\alpha$, and Fourier inversion on $\mathbb{Z}_{p^\alpha}$ gives $t_h=p^{-\alpha}\sum_{j=0}^{p^\alpha-1}\hat c(j)\omega^{-jh}=p^{-\alpha}\hat c(0)=|C|/p^\alpha$.
\end{proof}

\begin{lem}\label{lem:zerosum}
Let $p$ be a prime, $\alpha\ge1$, $\omega$ a primitive $p^\alpha$-th root of unity, $b_0\in\mathbb{Z}$, and $m_1,\dots,m_r\in\{1,\dots,p^\alpha-1\}$ (repetitions allowed). If $b_0+\sum_{t=1}^r\omega^{m_t}=0$, then for each $1\le j\le p-1$ the number $jp^{\alpha-1}$ occurs exactly $b_0$ times among $m_1,\dots,m_r$. In particular $b_0\ge0$, and if $b_0\ge1$ then $r\ge(p-1)b_0$, with equality only when $\{m_t\}^r_{t=1}$ consists of $p^{\alpha-1},2p^{\alpha-1},\dots,(p-1)p^{\alpha-1}$, each taken exactly $b_0$ times.
\end{lem}
\begin{proof}
Let $f(x)=b_0+\sum_tx^{m_t}\in\mathbb{Z}[x]$ and $\Phi(x)=\sum_{j=0}^{p-1}x^{jp^{\alpha-1}}$, the $p^\alpha$-th cyclotomic polynomial, which is the minimal polynomial of $\omega$ over $\mathbb{Q}$ (see \cite[Chapter~13]{Dummit}). Dividing $f$ by the monic polynomial $\Phi$ in $\mathbb{Z}[x]$ leaves a remainder of smaller degree that vanishes at $\omega$, hence is $0$; so $f=\Phi\Psi$ with $\Psi\in\mathbb{Z}[x]$. Because $\deg f\le p^\alpha-1$ and $\deg\Phi=(p-1)p^{\alpha-1}$, we have $\deg\Psi\le p^{\alpha-1}-1$; write $\Psi=\sum_{i=0}^{p^{\alpha-1}-1}a_ix^i$. The exponents $jp^{\alpha-1}+i$ with $0\le j\le p-1$ and $0\le i<p^{\alpha-1}$ are pairwise distinct, so the coefficient of $x^{jp^{\alpha-1}}$ in $\Phi\Psi$ is $a_0$ for every $0\le j\le p-1$. For $j=0$ this gives $a_0=b_0$. For $j\ge1$, the coefficient of $x^{jp^{\alpha-1}}$ in $f$ is the number of $t$ with $m_t=jp^{\alpha-1}$ (recall $m_t\ge1$), and it equals $a_0=b_0$.
\end{proof}

\begin{lem}\label{lem:obstruction}
Let $(a,b)\in\{(1,3),(3,1),(0,2),(2,1),(1,2)\}$, and let $p$, $d=b-a$ and $c_0$ be as follows.
\[
\begin{array}{c|ccccc}
(a,b) & (1,3) & (3,1) & (0,2) & (2,1) & (1,2)\\\hline
p & 2 & 2 & 3 & 3 & 5\\
d & 2 & -2 & 2 & -1 & 1\\
c_0 & 1 & 1 & 1 & 1 & 2
\end{array}
\]
If $\Gamma$ has an $(a,b)$-regular set, then $p\mid n$, and for $\alpha=v_p(n)$ and a primitive $p^\alpha$-th root of unity $\omega$ the number $\lambda_j$ of Lemma~\ref{lem:engine} vanishes for some $1\le j\le p^\alpha-1$.
\end{lem}
\begin{proof}
Lemma~\ref{lem:count} gives $|C|=bn/(4+d)=c_0n/p$ in each of the five cases. Since $p\nmid c_0$, we get $p\mid n$. If all $\lambda_j$ were nonzero, Lemma~\ref{lem:engine} would give $t_h=c_0n/p^{\alpha+1}\in\mathbb{Z}$, hence $p^{\alpha+1}\mid n$, contradicting $\alpha=v_p(n)$.
\end{proof}

\begin{lem}\label{lem:constructions}
Let $q\mid n$, let $\emptyset\ne R\subsetneq\mathbb{Z}_q$, and put $C_R=\{x^m:\ m\bmod q\in R\}$. Let $M$ be the multiset $\{\pm\bar k,\pm\bar\ell\}$ of residues of the exponents in $\mathbb{Z}_q$. If, for every $r\in\mathbb{Z}_q$, exactly $a$ elements of the multiset $r+M$ lie in $R$ when $r\in R$, and exactly $b$ when $r\notin R$, then $C_R$ is $(a,b)$-regular. This applies in the following cases.
\begin{center}
\begin{tabular}{c|c|c|c|l}
$(a,b)$ & $q$ & $R$ & $M$ & hypothesis on $k,\ell$\\\hline
$(0,2)$ & $3$ & $\{0\}$ & $\{1,2,1,2\}$ & $k\equiv\ell\equiv1\pmod 3$\\
$(2,1)$ & $3$ & $\{0\}$ & $\{1,2,0,0\}$ & $\{k,\ell\}\equiv\{1,0\}\pmod 3$\\
$(1,2)$ & $5$ & $\{0,1\}$ & $\{1,2,3,4\}$ & $\{k,\ell\}\equiv\{1,2\}\pmod 5$\\
$(1,3)$ & $4$ & $\{0,1\}$ & $\{1,3,2,2\}$ & $k$ odd, $\ell\equiv2\pmod 4$\\
$(3,1)$ & $4$ & $\{0,1\}$ & $\{1,3,0,0\}$ & $k$ odd, $\ell\equiv0\pmod 4$
\end{tabular}
\end{center}
\end{lem}
\begin{proof}
The exponents of the neighbors $x^{m\pm k},x^{m\pm\ell}$ of $x^m$ have residues $(m\bmod q)+M$ modulo $q$, counted with multiplicity, which gives the first assertion. For the five rows the multisets $r+M$ for $r=0,\dots,q-1$ and the numbers of their elements in $R$ are:
\begin{itemize}
\item $(0,2)$: $\{1,2,1,2\},\{2,0,2,0\},\{0,1,0,1\}$; the counts are $0,2,2$.
\item $(2,1)$: $\{1,2,0,0\},\{2,0,1,1\},\{0,1,2,2\}$; the counts are $2,1,1$.
\item $(1,2)$: $\{1,2,3,4\},\{2,3,4,0\},\{3,4,0,1\},\{4,0,1,2\},\{0,1,2,3\}$; the counts are $1,1,2,2,2$.
\item $(1,3)$: $\{1,3,2,2\},\{2,0,3,3\},\{3,1,0,0\},\{0,2,1,1\}$; the counts are $1,1,3,3$.
\item $(3,1)$: $\{1,3,0,0\},\{2,0,1,1\},\{3,1,2,2\},\{0,2,3,3\}$; the counts are $3,3,1,1$.
\end{itemize}
In each row the counts for $r\in R$ equal $a$ and those for $r\notin R$ equal $b$.
\end{proof}

\subsection{$(1,3)$- and $(3,1)$-regular sets}

Here $p=2$ and $d=\pm2$. For a primitive $2^\alpha$-th root of unity $\omega$ and an integer $w$ put $x(w)=\omega^w+\omega^{-w}=2\cos(2\pi w/2^\alpha)$. For residues $0\le w<2^\alpha$ we have $x(0)=2$, $x(2^{\alpha-1})=-2$, and for $\alpha\ge2$ we have $x(w)=0$ exactly when $v_2(w)=\alpha-2$; for all other $w$, $|x(w)|<2$ and $x(w)\ne0$.

\begin{prop}\label{prop:p2}
Let $\alpha\ge1$, let $\omega$ be a primitive $2^\alpha$-th root of unity, $d\in\{2,-2\}$, $k$ odd and $\ell$ a positive integer, and put $\lambda_j=d+x(jk)+x(j\ell)$ for $1\le j\le 2^\alpha-1$. If $\ell$ is even write $\beta=v_2(\ell)\ge1$. Then $\lambda_j\ne0$ for all $j$ in each of the following situations:
\begin{enumerate}
\item[(a)] $\ell$ is odd;
\item[(b)] $\ell$ is even and $\alpha=1$;
\item[(c)] $\ell$ is even, $\alpha\ge2$, and either $d=2$ and $\beta\ge2$, or $d=-2$ and $\beta=1$.
\end{enumerate}
\end{prop}
\begin{proof}
Fix $j$, let $s=v_2(j)\in\{0,\dots,\alpha-1\}$, and let $u,v\in\{0,\dots,2^\alpha-1\}$ be the residues of $jk$ and $j\ell$, so $\lambda_j=d+x(u)+x(v)$. As $k$ is odd, $u\ne0$ and $v_2(u)=s$. If $\ell$ is odd, likewise $v\ne0$ and $v_2(v)=s$. If $\ell$ is even, then $v_2(j\ell)=s+\beta$; so $v=0$ if $s+\beta\ge\alpha$, and $v_2(v)=s+\beta$ otherwise.

\emph{Case $s=\alpha-1$.} Here $u=2^{\alpha-1}$ and $x(u)=-2$. If $\ell$ is odd then also $v=2^{\alpha-1}$ and $\lambda_j=d-4\ne0$. If $\ell$ is even then $s+\beta\ge\alpha$, so $v=0$ and $\lambda_j=d-2+2=d\ne0$. For $\alpha=1$ this is the only case, which proves (b).

\emph{Case $s=\alpha-2$} ($\alpha\ge2$). Here $x(u)=0$. If $\ell$ is odd, $x(v)=0$ too and $\lambda_j=d\ne0$. If $\ell$ is even and $\beta=1$, then $v=2^{\alpha-1}$ and $\lambda_j=d-2$; if $\beta\ge2$, then $v=0$ and $\lambda_j=d+2$. Under (c) these values are $\pm4$.

\emph{Case $s\le\alpha-3$.} Here $x(u)\ne0$ and $|x(u)|<2$. Suppose first that $v_2(v)\le\alpha-3$, which is automatic when $\ell$ is odd. Then neither $u$ nor $v$ is $0$ or $2^{\alpha-1}$, so the four exponents $u,\,2^\alpha-u,\,v,\,2^\alpha-v$ lie in $\{1,\dots,2^\alpha-1\}\setminus\{2^{\alpha-1}\}$. If $\lambda_j=0$, Lemma~\ref{lem:zerosum} (with $p=2$, $b_0=d$) says that $2^{\alpha-1}$ occurs exactly $d$ times among them; it occurs never, and $d\ne0$, a contradiction. Suppose now that $v_2(v)\ge\alpha-2$. Then $x(v)\in\{2,-2,0\}$, so $d+x(v)\in\{4,0,2\}$ for $d=2$ and $d+x(v)\in\{0,-4,-2\}$ for $d=-2$. If $\lambda_j=0$, then $x(u)=-(d+x(v))\in\{0,\pm2,\pm4\}$, which is impossible because $x(u)\ne0$ and $|x(u)|<2$.
\end{proof}

\begin{proof}[Proof of Cases 11 and 15 of Theorem~\ref{thm:main}]
Let $d=2$ for Case 11 and $d=-2$ for Case 15. In every part of these two cases at least one of $k,\ell$ is odd; we assume it is $k$.

\emph{Existence.} Let $4\mid n$, and let $\ell\equiv2\pmod4$ in Case 11, respectively $\ell\equiv0\pmod 4$ in Case 15. Lemma~\ref{lem:constructions} gives a regular set. This is the ``if'' direction of part (i).

\emph{Non-existence.} Suppose $\Gamma$ has an $(a,b)$-regular set. By Lemma~\ref{lem:obstruction}, $n$ is even and some $\lambda_j$ vanishes, where $\alpha=v_2(n)$. If $k,\ell$ are both odd, this contradicts Proposition~\ref{prop:p2}(a), which proves part (iii) of both cases. In part (ii) of Case 11 we have $\ell\equiv0\pmod4$, and in part (ii) of Case 15 we have $\ell\equiv2\pmod 4$. If $\alpha=1$, Proposition~\ref{prop:p2}(b) gives a contradiction. If $\alpha\ge2$, then $\beta\ge2$ in Case 11 and $\beta=1$ in Case 15, and Proposition~\ref{prop:p2}(c) (with $d=2$, respectively $d=-2$) gives a contradiction. Finally, in part (i) ($\ell\equiv2\pmod4$ in Case 11, $\ell\equiv0\pmod4$ in Case 15) the case $\alpha=1$ is excluded by Proposition~\ref{prop:p2}(b), so $\alpha\ge2$, that is, $4\mid n$. This is the ``only if'' direction of (i).
\end{proof}

\subsection{$(0,2)$-, $(2,1)$- and $(1,2)$-regular sets}

Now $p\in\{3,5\}$. For a primitive $p^\alpha$-th root of unity $\omega$ and $1\le j\le p^\alpha-1$, let $u,v\in\{0,\dots,p^\alpha-1\}$ be the residues of $jk$ and $j\ell$, so that
\begin{equation}\label{eq:lambda-odd}
\lambda_j=d+\omega^{u}+\omega^{p^\alpha-u}+\omega^{v}+\omega^{p^\alpha-v}.
\end{equation}
If $p\nmid k$ then $u\ne0$, since $p^\alpha\mid jk$ would force $p^\alpha\mid j$; similarly $p\nmid\ell$ implies $v\ne0$.

\begin{prop}\label{prop:podd}
Let $p$ be an odd prime, $\alpha\ge1$ and $d\in\mathbb{Z}$. Then $\lambda_j\ne0$ for all $1\le j\le p^\alpha-1$ in each of the following situations:
\begin{enumerate}
\item[(a)] $p\nmid k\ell$ and $d<0$;
\item[(b)] $(p,d)\in\{(3,2),(5,1)\}$, $p\nmid k$ and $p\mid\ell$;
\item[(c)] $(p,d)=(5,1)$, $5\nmid k\ell$ and $k\equiv\ell\pmod5$.
\end{enumerate}
\end{prop}
\begin{proof}
Fix $j$ and put $\rho=p^{\alpha-1}$.

(a) Now $u,v\ne0$, so all four exponents in \eqref{eq:lambda-odd} lie in $\{1,\dots,p^\alpha-1\}$. If $\lambda_j=0$, Lemma~\ref{lem:zerosum} would force $d\ge0$.

(b) Now $u\ne0$. If $v=0$, then $\lambda_j=d+2+2\cos(2\pi u/p^\alpha)\ge d>0$. Let $v\ne0$ and suppose $\lambda_j=0$. Lemma~\ref{lem:zerosum} with $b_0=d$ and $r=4=(p-1)d$ is in the equality case, so every exponent in \eqref{eq:lambda-odd} is a multiple of $\rho$. In particular $\rho\mid u$, that is $p^{\alpha-1}\mid jk$, hence $p^{\alpha-1}\mid j$. Since $p\mid\ell$, this gives $p^\alpha\mid j\ell$, that is $v=0$, a contradiction.

(c) Now $u,v\ne0$. Suppose $\lambda_j=0$. Lemma~\ref{lem:zerosum} with $b_0=1$ and $r=4=(p-1)b_0$ shows that the exponents $u,\,p^\alpha-u,\,v,\,p^\alpha-v$ are $\rho,2\rho,3\rho,4\rho$ in some order; in particular $u\ne v$. Also $\rho\mid u$ gives $\rho\mid j$, say $j=e\rho$. Then $u\equiv ek\rho$ and $v\equiv e\ell\rho\pmod{5^\alpha}$ depend only on $ek$ and $e\ell$ modulo $5$, which coincide because $k\equiv\ell\pmod5$. Hence $u=v$, a contradiction.
\end{proof}

\begin{proof}[Proof of Cases 3, 10 and 13 of Theorem~\ref{thm:main}]
Take $(p,d)=(3,2)$ for $(0,2)$, $(3,-1)$ for $(2,1)$ and $(5,1)$ for $(1,2)$, as in Lemma~\ref{lem:obstruction}.

\emph{Existence.} Parts 3(i), 13(i) and 10(i) say that $\Gamma$ has a regular set if and only if $p\mid n$. By Lemma~\ref{lem:obstruction} the condition $p\mid n$ is necessary. If $p\mid n$ and the congruence hypothesis holds, Lemma~\ref{lem:constructions} provides a regular set.

\emph{Non-existence.} Suppose a regular set exists. By Lemma~\ref{lem:obstruction}, $p\mid n$ and some $\lambda_j$ vanishes, where $\alpha=v_p(n)$. This contradicts Proposition~\ref{prop:podd} in the remaining parts:
\begin{itemize}
\item 3(ii): exactly one of $k,\ell$ is divisible by $3$; interchange them so that $3\mid\ell$, $3\nmid k$, and use \ref{prop:podd}(b) with $(p,d)=(3,2)$;
\item 13(ii): $k\equiv\ell\equiv1\pmod3$ and $d=-1<0$, so \ref{prop:podd}(a) applies;
\item 10(ii): $k\equiv\ell\equiv1\pmod5$; use \ref{prop:podd}(c);
\item 10(iii): exactly one of $k,\ell$ is divisible by $5$; interchange so that $5\mid\ell$ and use \ref{prop:podd}(b) with $(p,d)=(5,1)$.
\end{itemize}
\end{proof}


\end{document}